\RequirePackage{plautopatch}
\documentclass[dvipdfmx,reqno]{amsart}
\usepackage{amssymb,amsmath,amsthm,bm}
\usepackage{enumerate,ascmac}

\usepackage{accents}

\makeatletter
\def\@settitle{\begin{center}\baselineskip14\p@\relax\bfseries{\large\@title}\thispagestyle{empty}\end{center}}
\def\@setauthors{%
  \begingroup
  \def\thanks{\protect\thanks@warning}%
  \trivlist
  \centering\footnotesize \@topsep30\p@\relax
  \advance\@topsep by -\baselineskip
  \item\relax
  \author@andify\authors
  \def\\{\protect\linebreak}%
  {\authors}%
  \ifx\@empty\contribs
  \else
    ,\penalty-3 \space \@setcontribs
    \@closetoccontribs
  \fi
  \endtrivlist
  \endgroup
}
\def\maketitle{\par
  \@topnum\z@ 
  \@setcopyright
  \thispagestyle{firstpage}
  \ifx\@empty\shortauthors \let\shortauthors\shorttitle
  \else \andify\shortauthors
  \fi
  \@maketitle@hook
  \begingroup
  \@maketitle
  \toks@\@xp{\shortauthors}\@temptokena\@xp{\shorttitle}%
  \toks4{\def\\{ \ignorespaces}}
  \edef\@tempa{%
    \@nx\markboth{\the\toks4
      \@nx
      {\the\toks@}}{\the\@temptokena}}%
  \@tempa
  \endgroup
  \c@footnote\z@
  \@cleartopmattertags
}
\makeatother

\usepackage{tikz}
\usetikzlibrary{intersections,calc,arrows.meta,patterns}
\usepackage{color}
\newtheorem{Theorem}{Theorem}[section]
\newtheorem{Proposition}[Theorem]{Proposition}
\newtheorem{Lemma}[Theorem]{Lemma}

\theoremstyle{definition}

\newtheorem{Remark}[Theorem]{Remark}

\newcommand{\pmat}[1]{\begin{pmatrix}#1\end{pmatrix}}

\newcommand{\smat}[1]{\left(\begin{smallmatrix}#1\end{smallmatrix}\right)}
\newcommand{\coupling}[2]{\left(\,#1\,\middle|\,#2\,\right)}
\newcommand{\dtilde}[1]{\accentset{\approx}{#1}}

\newcommand{\bs}[1]{\boldsymbol{#1}}

\newcommand{\set}[2]{\left\{#1;\,#2\right\}}
\newcommand{\setbig}[2]{\bigl\{#1;\,#2\bigr\}}
\newcommand{\norm}[1]{\left\lVert#1\right\rVert}
\newcommand{\innV}[2]{\left\langle\,#1\,\middle|\,#2\,\right\rangle}
\newcommand{\R}{\mathbb{R}}

\newcommand{\tm}{\triangle}
\newcommand{\Clanprod}[2]{#1\,\tm\,#2}
\newcommand{\dtm}{\rotatebox[origin=c]{180}{$\tm$}}
\newcommand{\dClanprod}[2]{#1\,\dtm\,#2}
\newcommand{\ul}[1]{\underline{#1}}
\newcommand{\ol}[1]{\overline{#1}}
\newcommand{\newone}[1]{#1'}
\makeatletter

\@addtoreset{equation}{section}
\makeatother

\begin{document}
\title
{An example of homogeneous cones whose basic relative invariant has maximal degree}
\author{Hideto Nakashima}
\address{The Institute of Statistical Mathematics\\ 
Midori-cho 10-3, Tachikawa, Tokyo 190-8562, 
Japan}
\email{hideto@ism.ac.jp}
\date{}
\begin{abstract}
It is known that degrees of basic relative invariants of homogeneous open convex cones of rank $r$ are less than or equal to $2^{r-1}$.
In this article,
we show that there exists a homogeneous cone of rank $r$
one of whose basic relative invariants has degree $2^{r-1}$.
The main idea for this is to 
construct such a homogeneous cone inductively to have specific structure constants
which enable us to calculate degrees of its basic relative invariants.
We study homogeneous cones of rank $3$ in detail
in order to see
non-triviality of the existence of homogeneous cones with given structure constants.
%
\end{abstract}
\maketitle
\section{Introduction}
In this article,
we consider a problem raised in a paper \cite{GIL} on relationship between ranks of homogeneous open convex cones and degrees of their basic relative invariants.
This problem is related to hyperbolic polynomials connected with many important classes of optimization problems such as linear programs,
semi-definite programs and, more generally, symmetric cone programming.
Degrees of hyperbolic polynomials play an important role in interior-point methods of which the degrees are used to bound the worst-case iteration complexity
(cf.\ \cite{GT98,Gul97}).
The paper~\cite{GIL} shows that 
any homogeneous open convex cone of rank $r$ satisfies $d\le 2^{r-1}$ 
where $d$ is a degree of its basic relative invariant,
and they asked that this evaluation can be improved or not.
We give an answer to this problem negatively by giving an example of homogeneous open convex cone satisfying $d=2^{r-1}$,
and hence this upper bound is strict.

\begin{Theorem}
\label{theo}
For any rank $r$,
there exists a homogeneous open convex cone one of whose basic relative invariants has degree $2^{r-1}$.
\end{Theorem}

Let $\Omega$ be a homogeneous open convex cone of rank $r$ in a finite dimensional vector space $V$.
In this article,
we employ a matrix realization of homogeneous cones due to Ishi~\cite{Ishi2006}.
Let us take a collection $\set{\mathcal{V}_{kj}}{1\le j<k\le r}$ of vector spaces 
corresponding to $\Omega$
(see \eqref{def:V} for detail).
Since we have an algorithm for calculating degrees of basic relative invariants of $\Omega$ by using structure constants $\dim \mathcal{V}_{kj}$ (Lemma~\ref{lemma:alg}),
if a homogeneous cone $\Omega\subset V$ satisfies
\begin{equation}\label{eq:property}
\dim \mathcal{V}_{kj}=2^{k-j}\quad(1\le j< k\le r),
\end{equation}
then $\Omega$ can be an example of Theorem~\ref{theo}.
However, it is not clear whether such a homogeneous cone exists or not
(cf.\ Remark~\ref{rem:1}).
In order to explain its non-triviality,
we discuss homogeneous cones of rank $3$ in detail in Section \ref{sect:hom3}.
In particular, we will see that homogeneous cones must be related 
to the Hurwitz problem on quadratic forms (see \eqref{eq:Hurwitz}).
We shall show the existence of homogeneous cones satisfying \eqref{eq:property}
by using a method of constructing a new homogeneous cone 
from a given homogeneous cone (cf.\ \cite{N2014}).

This article is organized as follows.
Section~\ref{sect:prelim} collects definitions and basic properties of homogeneous cones.
In particular, we state an algorithm for calculating degrees of basic relative invariants of homogeneous cones in Lemma~\ref{lemma:alg}.
The proof of the main theorem, Theorem~\ref{theo} is given in Section~\ref{sect:proof}.
In Section~\ref{sect:hom3},
we consider homogeneous cones of rank $3$ in order to explain non-triviality of the existence of homogeneous cones with given $\dim \mathcal{V}_{kj}$.
We also classify irreducible homogeneous cones of rank $3$ with respect to degrees of basic relative invariants in Proposition~\ref{prop:cone-rsn}.

\section*{Acknowledgements}
The author is grateful to Professor Bruno F.\ Louren{\c{c}}o for the encouragement in writing this article.
He also thanks Professor Hideyuki Ishi for valuable comments.
The author was supported by the Grant-in-Aid of scientific research of JSPS No.\ 23K03061.

\section{Preliminaries}
\label{sect:prelim}
Let $V$ be a finite dimensional real vector space 
and 
$\Omega$ an open convex cone in $V$ containing no entire line.
The cone $\Omega$ is said to be homogeneous
if
the group $GL(\Omega)=\set{g\in GL(V)}{g\Omega=\Omega}$
acts on $\Omega$ transitively.
A homogeneous open convex cone is said to be irreducible
if
there exist no non-trivial subspaces $V_1,V_2\subset V$
and homogeneous open convex cones $\Omega_j\subset V_j$ $(j=1,2)$
such that
$V$ is a direct sum of $V_1$ and $V_2$,
and $\Omega=\Omega_1+\Omega_2$.
In this article,
we always assume that an open convex cone $\Omega$ is homogeneous and irreducible,
and call it a homogeneous cone for short.
A general theory of homogeneous cones is established by Vinberg~\cite{Vinberg}.
For describing homogeneous cones,
we employ the matrix realization of homogeneous cones due to Ishi (\cite[\S3.1]{Ishi2006})
since it requires no preliminary knowledge on the theory of convex cones.

Let $N=n_1+\cdots+n_r$ ($n_1,\dots,n_r>0$) be 	an ordered partition of
a positive integer $N$.
Let $\mathcal{V}$ be a collection of vector spaces $\mathcal{V}_{kj}\subset\mathrm{Mat}(n_k,n_j;\,\R)$ $(1\le j< k\le r)$
satisfying the following conditions.
\begin{enumerate}
\item[(V1)] $X_{kj}\in \mathcal{V}_{kj}$, $X_{ji}\in \mathcal{V}_{ji}$ $\Rightarrow$ $X_{kj}X_{ji}\in \mathcal{V}_{ki}$ $(1\le i<j<k\le r)$,
\item[(V2)] $X_{ki}\in \mathcal{V}_{ki}$, $X_{ji}\in \mathcal{V}_{ji}$ $\Rightarrow$ $X_{ki}{}^{\,t\!}X_{ji}\in \mathcal{V}_{kj}$
$(1\le i<j<k\le r)$,
\item[(V3)] $X_{kj}\in \mathcal{V}_{kj}$ $\Rightarrow$ $X_{kj}{}^{\,t\!}X_{kj}\in \R I_{n_k}$.
\end{enumerate}
Here, $I_m$ is the unit matrix of size $m$.
We introduce a subgroup $H$ of $GL(N,\R)$ and a linear subspace $V$ of $\mathrm{Sym}(N,\R)$ by
\[
H=\set{h=\pmat{
	t_{11}I_{n_1}&0&\cdots&0\\
	T_{21}&t_{22}I_{n_2}&\ddots&\vdots\\ 
	\vdots&\ddots&\ddots&0\\
	T_{r1}&T_{r2}&\cdots&t_{rr}I_{n_r}
}}{\begin{array}{l}t_{11},\dots,t_{rr}\ne 0\\ T_{kj}\in \mathcal{V}_{kj}\\
(1\le j<k\le r)\end{array}}
\]
and
\begin{equation}\label{def:V}
V=\set{x=\pmat{
	x_{11}I_{n_1}&{}^{t\!}X_{21}&\cdots&{}^{t\!}X_{r1}\\
	X_{21}&x_{22}I_{n_2}&\ddots&{}^{t\!}X_{r2}\\ 
	\vdots&\ddots&\ddots&\vdots\\
	X_{r1}&X_{r2}&\cdots&x_{rr}I_{n_r}
}}{\begin{array}{l}x_{11},\dots,x_{rr}\in\R\\ X_{kj}\in \mathcal{V}_{kj}\\
(1\le j<k\le r)\end{array}}.
\end{equation}
According to the conditions (V1)--(V3),
a map $\rho$ defined by
\[
\rho(h)x:=hx{}^{\,t\!}h\quad(h\in H,\ x\in V)
\]
is a rational representation of $H$ on $V$.
Then, an open convex cone $\Omega_\mathcal{V}$ defined by
\[
\Omega_\mathcal{V}:=\set{x\in V}{\text{$x$ is positive definite}}
\]
is a homogeneous cone of rank $r$ because
$H$ acts on $\Omega_\mathcal{V}$ transitively through $\rho$.
Conversely,
any homogeneous cone of rank $r$ can be realized as $\Omega_\mathcal{V}$ for some $\mathcal{V}$
(cf.\ Ishi~\cite{Ishi2006}).
Thus, in what follows,
we assume that a homogeneous cone $\Omega$ is realized as in this matrix form
unless otherwise mentioned.

We equip $V$ with an inner product 
\begin{equation}\label{def:innV}
\innV{x}{y}_V=\sum_{i=1}^rx_iy_i+2\sum_{1\le j<k\le r}\innV{X_{kj}}{Y_{kj}}_{kj}\quad(x,y\in V),
\end{equation}
where $\innV{\cdot}{\cdot}_{kj}$ $(1\le j< k\le r)$ is an inner product of $\mathcal{V}_{kj}$ defined by
\begin{equation}\label{def:innV kj}
\frac12\left(X_{kj}{}^{\,t\!}Y_{kj}+Y_{kj}{}^{\,t\!}X_{kj}\right)
=
\innV{X_{kj}}{Y_{kj}}_{kj}I_{n_k}\quad(X_{kj},Y_{kj}\in \mathcal{V}_{kj}).
\end{equation}
Note that this inner product is not equal to the trace inner product $\mathrm{tr}(xy)$
unless $n_1=\cdots=n_r=1$.
The dual cone $\Omega^*$ of $\Omega$ is defined to be
\[
\Omega^*:=\set{y\in V}{\innV{x}{y}_V>0\textrm{ for all }x\in\overline{\Omega}\setminus\{0\}}.
\]
Here,
$\overline{\Omega}$ is the closure of $\Omega$ in $V$.

The basic relative invariants $\Delta_1,\dots,\Delta_r$ of $\Omega$ 
are irreducible polynomials on $V$ satisfying
\[
\Delta_j\bigl(\rho(h)x\bigr)=\chi_j(h)\Delta_j(x)\quad(h\in H,\ x\in V),
\]
where $\chi_j\colon H\to\R^{\times}$ is a rational character of $H$,
and $\Omega$ can be described as
\[
\Omega=\set{x\in V}{\Delta_1(x)>0,\dots,\Delta_r(x)>0}.
\]
Since $H$ is lower triangular,
the basic relative invariants are obtained as irreducible factors of
the left upper corner principal minors $\det^{[\ell]}x$ $(\ell=1,\dots,N)$ of a matrix $x\in V$.
In this article, we determine 
the numbering of $\Delta_1,\dots,\Delta_r$
by taking irreducible factors of $\det^{[\ell]}x$ along $\ell=1,\dots,N$
so that we always have $\Delta_1(x)=x_{11}$.
Associated with $\Delta_j(x)$ $(j=1,\dots,r)$,
we introduce a matrix $\sigma:=(\sigma_{jk})_{1\le j,k\le r}$ by
\[
\Delta_j(x)=x_{11}^{\sigma_{j1}}\cdots x_{rr}^{\sigma_{jr}}
\quad (x=\mathrm{diag}(x_{11}I_{n_1},\ldots,x_{rr}I_{n_r})\in \Omega).
\]
Then,
since now we fix the numbering of basic relative invariants, 
$\sigma$ is a lower triangular matrix of integer entries
with ones on the main diagonal,
and it satisfies
\begin{equation}\label{eq:degrees}
\pmat{\deg\Delta_1\\ \vdots\\ \deg\Delta_r}
=
\sigma\pmat{1\\ \vdots\\ 1}.
\end{equation}
In the previous paper~\cite{N2014},
we give an algorithm for calculating $\sigma$
by using structure constants $d_{kj}:=\dim\mathcal{V}_{kj}$ $(1\le j<k\le r)$.
Set $\bm d_i:={}^t(0,\dots,0,d_{i+1,i},\dots,d_{ri})\in\R^r$ for $i=1,\dots,r-1$.

\begin{Lemma}[\cite{N2014}, see also {\cite[Lemma 1.1]{N2018}}]
\label{lemma:alg}
For $i=1,\dots,r-1$,
one defines $\bm l^{(j)}_i={}^t(l^{(j)}_{1i},\dots,l^{(j)}_{ri})\in\R^r$ $(j=i,\dots,r-1)$ inductively by
\[
\bm l^{(i)}_i:=\bm d_i,\qquad
\bm l^{(k)}_i:=
\begin{cases}
\bm l^{(k-1)}_i-\bm d_k	&(l^{(k-1)}_{ki}>0),\\
\bm l^{(k-1)}_i			&(l^{(k-1)}_{ki}=0)
\end{cases}
\quad(k=i+1,\dots,r-1).
\]
Let $\bm \varepsilon^{[i]}={}^t(\varepsilon_{i+1,i},\dots,\varepsilon_{ri})\in\{0,1\}^{r-i}$ $(i=1,\dots,r-1)$ be a vector defined by
\[
\varepsilon_{ji}=\begin{cases}
1&(\text{if }\quad l^{(r-1)}_{ji}>0),\\
0&(\text{if }\quad l^{(r-1)}_{ji}=0).
\end{cases}
\]
Then, $\sigma$ is given as
\[
\sigma=\mathcal{E}_{r-1}\mathcal{E}_{r-2}\cdots\mathcal{E}_1,\quad
\mathcal{E}_i=\pmat{I_{i-1}&0&0\\0&1&0\\0&\bm \varepsilon^{[i]}&I_{r-i}}\quad
(i=1,\dots,r-1).
\]
\end{Lemma}

\begin{Remark}
\label{rem:0}
If we set $d_{kj}=2^{k-j}$ $(1\le j< k\le r)$ as in \eqref{eq:property},
then it is easily verified that
$\varepsilon_{ji}=1$ for any $1\le i<j\le r$ so that
we obtain
\[
\sigma=\bigl(\sigma_{ij}\bigr)_{1\le i,j\le r}\qquad
\sigma_{ij}=\begin{cases}
	0			&(i<j),\\
	1			&(i=j),\\
	2^{i-j-1}	&(i>j).
\end{cases}
\]
By \eqref{eq:degrees},
if such a homogeneous cone exists, then we have
\[
\deg\Delta_r=2^{r-2}+2^{r-3}+\cdots+2^0+1=2^{r-1}.
\]
However, the existence of homogeneous cones with given $\dim \mathcal{V}_{kj}$ is 
not at all trivial
(see Remark~\ref{rem:1}).
We shall see this via homogeneous cones of rank $3$ in Section~\ref{sect:hom3}.
%
\end{Remark}

\section{Proof of Theorem~\ref{theo}}
\label{sect:proof}

Let us keep all notations in the previous section,
and we shall give a proof to Theorem~\ref{theo}.
Let us define a subspace $\newone{V}$ of $\mathrm{Sym}(N+2n_1,\R)$ from data $\mathcal{V}$ by
\[
\newone{V}
=
\set{\left(\begin{array}{@{}cc|cccc@{}}
	x_{00}I_{n_1}&0		&y_{10}I_{n_1}&{}^{t\!}Y_{21}&\cdots&{}^{t\!}Y_{r1}\\
	0&x_{00}I_{n_1}		&z_{10}I_{n_1}&{}^{t\!}Z_{21}&\cdots&{}^{t\!}Z_{r1}\\
	\hline
	y_{10}I_{n_1}&z_{10}I_{n_1}
					&x_{11}I_{n_1}&{}^{t\!}X_{21}&\cdots&{}^{t\!}X_{r1}\\
	Y_{21}&Z_{21}	&X_{21}&x_{22}I_{n_2}&\ddots&{}^{t\!}X_{r2}\\
	\vdots&\vdots	&\vdots&\ddots&\ddots&\vdots\\
	Y_{r1}&Z_{r1}	&X_{r1}&X_{r2}&\cdots&x_{rr}I_{n_r}
\end{array}\right)}{
\begin{array}{l}
x_{ii}\in\R\\
(i=0,1,\dots,r)\\
y_{10},z_{10}\in\R\\
X_{kj}\in \mathcal{V}_{kj}\\
Y_{k1},Z_{k1}\in \mathcal{V}_{k1}\\
(1\le j<k\le r)
\end{array}}.
\]

\begin{Lemma}
An open convex cone $\newone{\Omega}:=\set{x\in V'}{\text{$x$ is positive definite}}$ is a homogeneous cone of rank $r+1$.
\end{Lemma}
\begin{proof}
Let us set 
\begin{align*}
n_0&:=2n_1,\qquad
\newone{\mathcal{V}}_{kj}:=\mathcal{V}_{kj}\quad(1\le j<k\le r),\\
\newone{\mathcal{V}}_{10}&:=\set{\pmat{y_{10}I_{n_1}&z_{10}I_{n_1}}}{y_{10},z_{10}\in\R}
\subset\mathrm{Mat}(n_1,n_0;\,\R),\\
\newone{\mathcal{V}}_{k0}&:=
\set{\pmat{Y_{k1}&Z_{k1}}}{Y_{k1},Z_{k1}\in \mathcal{V}_{k1}}
\subset\mathrm{Mat}(n_k,n_0;\,\R)
\quad(k=2,\dots,r).
\end{align*}
Then, it is enough to show that the set $\setbig{\newone{\mathcal{V}}_{kj}}{0\le j< k\le r}$
satisfies the conditions $(V1)$--$(V3)$.
Since its subset $\setbig{\newone{\mathcal{V}}_{kj}}{1\le j<k\le r}$ satisfies these conditions by definition,
we shall check it with respect to $\newone{\mathcal{V}}_{k0}$ $(k=1,\dots,r)$.

The case of $(V1)$. 
Let us take $X_{kj}\in \newone{\mathcal{V}}_{kj}$ and 
$X_{j0}\in \newone{\mathcal{V}}_{j0}$
($1\le j<k\le r$).
Then we have $X_{j0}=(Y_{j1}\,Z_{j1})$ for some $Y_{j1},Z_{j1}\in \mathcal{V}_{j1}$.
Here we regard $\mathcal{V}_{11}=\R I_{n_1}$.
Since 
\[X_{kj}Y_{j1},\ X_{kj}Z_{j1}\in \mathcal{V}_{k1}=\newone{\mathcal{V}}_{k1},\]
we obtain
\[
X_{kj}X_{j0}=\pmat{X_{kj}Y_{j1}&X_{kj}Z_{j1}}\in \newone{\mathcal{V}}_{k0}.
\]

The case of $(V2)$.
Let us take $X_{k0}\in\newone{\mathcal{V}}_{k0}$ and $X_{j0}\in\newone{\mathcal{V}}_{j0}$
$(1\le j<k\le r)$.
Then we have
$X_{k0}=(Y_{k1},\,Z_{k1})$ for some $Y_{k1},Z_{k1}\in \mathcal{V}_{k1}$
and
$X_{j0}=(Y_{j1},\,Z_{j1})$ for some $Y_{j1},Z_{j1}\in \mathcal{V}_{j1}$.
Since $Y_{k1}{}^{\,t\!}Y_{j1}$, $Z_{k1}{}^{\,t\!}Z_{j1}\in \mathcal{V}_{kj}$,
we obtain
\[
X_{k0}{}^{\,t\!}X_{j0}
=
\pmat{Y_{k1}&Z_{k1}}\pmat{{}^{t\!}Y_{j1}\\{}^{t\!}Z_{j1}}
=
Y_{k1}{}^{\,t\!}Y_{j1}+Z_{k1}{}^{\,t\!}Z_{j1}\in \mathcal{V}_{kj}=\newone{\mathcal{V}}_{kj}.
\]

The case of $(V3)$.
Let us take $X_{k0}\in\newone{\mathcal{V}}_{k0}$.
Then we have $X_{k0}=(Y_{k1},\,Z_{k1})$ for some $Y_{k1},Z_{k1}\in \mathcal{V}_{k1}$.
Since $Y_{k1}{}^{\,t\!}Y_{k1}$, $Z_{k1}{}^{\,t\!}Z_{k1}\in \R I_{n_k}$,
we obtain
\[
X_{k0}{}^{\,t\!}X_{k0}
=
\pmat{Y_{k1}&Z_{k1}}\pmat{{}^{t\!}Y_{k1}\\{}^{t\!}Z_{k1}}
=
Y_{k1}{}^{\,t\!}Y_{k1}+Z_{k1}{}^{\,t\!}Z_{k1}\in \R I_{n_k}.
\]

These calculations show that 
the set $\setbig{\newone{\mathcal{V}}_{kj}}{0\le j< k\le r}$
satisfies the conditions $(V1)$--$(V3)$,
and hence $\newone{\Omega}$ is a homogeneous cone.
\end{proof}

\begin{proof}[Proof of Theorem~\ref{theo}]
Now we can prove Theorem~\ref{theo}.
By construction of $\newone{\mathcal{V}}$, 
it satisfies
\[
\dim \newone{\mathcal{V}}_{k0}=2\dim \newone{\mathcal{V}}_{k1}\quad(k=1,\dots,r),
\]
and this construction can be proceeded again and again.
Therefore,
if we start from the one-dimensional homogeneous cone $\Omega_1=\R_{>0}$ and apply this construction repeatedly,
then the resultant is a homogeneous cone of rank $r$ with the property \eqref{eq:property}.
This shows the existence of a homogeneous cone satisfying \eqref{eq:property},
and hence we have proved Theorem~\ref{theo} from a discussion in Remark~\ref{rem:0}.
\end{proof}

\begin{Remark}
The following are the first three examples of homogeneous cones generated by this algorithm.
\[
\Omega_1=\R_{>0},\quad
\Omega_2=\set{x=\left(\begin{array}{cc|c}
	x_{00}&0&y_{10}\\
	0&x_{00}&z_{10}\\
	\hline
	y_{10}&z_{10}&x_{11}
\end{array}\right)}{x\gg0},
\]
\[
\Omega_3=\set{x=\left(\begin{array}{cccc|cc|c}
	x_{00}&0&0&0		&y_{10}&0&y_{20}^1\\
	0&x_{00}&0&0		&0&y_{10}&y_{20}^2\\
	0&0&x_{00}&0		&z_{10}&0&z_{20}^1\\
	0&0&0&x_{00}		&0&z_{10}&z_{20}^2\\
	\hline
	y_{10}&0&z_{10}&0	&x_{11}&0&x_{21}^1\rule{0pt}{9pt}\\
	0&y_{10}&0&z_{10}	&0&x_{11}&x_{21}^2\\
	\hline
	y_{20}^1&y_{20}^2
		&z_{20}^1&z_{20}^2
						&x_{21}^1&x_{21}^2&x_{33}\rule{0pt}{9pt}
\end{array}\right)}{x\gg0}.
\]
\end{Remark}

\section{Homogeneous cones of rank $3$}
\label{sect:hom3}

In this section, 
we consider homogeneous cones of rank $3$ in order to explain non-triviality of the existence of homogeneous cones with given $\dim \mathcal{V}_{kj}$.
We also give an explicit matrix realization of homogeneous cones of rank $3$
and classify them with respect to degrees of basic relative invariants.

Let $\Omega$ be a homogeneous cone of rank $3$ with $\mathcal{V}=\{\mathcal{V}_{21},\mathcal{V}_{31},\mathcal{V}_{32}\}$.
For brevity,
let us put
\[
r:=d_{32}=\dim \mathcal{V}_{32},\quad s:=d_{21}=\dim \mathcal{V}_{21},\quad n:=d_{31}=\dim \mathcal{V}_{31}.
\]
The condition $(V1)$ implies that
$X_{32}X_{21}\in \mathcal{V}_{31}$ ($X_{32}\in \mathcal{V}_{32}$, $X_{21}\in \mathcal{V}_{21}$)
and hence we have by $(V3)$
\[
\norm{X_{32}X_{21}}_{31}^2I_{n_3}
=
X_{32}X_{21}{}^{\,t\!}\bigl(X_{32}X_{21}\bigr)
=
\norm{X_{32}}_{32}^2\norm{X_{21}}_{21}^2I_{n_{3}}.
\]
Here, each $\norm{\cdot}_{kj}$ is a norm induced from the inner product $\innV{\cdot}{\cdot}_{kj}$ defined in \eqref{def:innV kj}.
Suppose that $r,s,n\ge 1$.
Then, if we take orthonormal bases $\{\bs{e}_{kj}^{(\nu)}\}_{\nu=1,\dots,\dim \mathcal{V}_{kj}}$ of $\mathcal{V}_{kj}$ with respect to these norms and if we write
\begin{equation}
\label{eq:coord}
X_{32}=\sum_{\nu=1}^rx_\nu\bs{e}_{32}^{(\nu)},\quad
X_{21}=\sum_{\nu=1}^sy_\nu\bs{e}_{21}^{(\nu)},\quad
X_{32}X_{21}=\sum_{\nu=1}^nz_\nu\bs{e}_{31}^{(\nu)},
\end{equation}
then each $z_\nu$ $(\nu=1,\dots,n)$ is a bilinear form in the $x_i$ and the $y_j$ with coefficients in $\R$, and we have
\begin{equation}
\label{eq:Hurwitz}
(x_1^2+\cdots+x_r^2)(y_1^2+\cdots+y_s^2)
=
z_1^2+\cdots+z_n^2.
\end{equation}
A triplet $(r,s,n)$ is said to be admissible if the equation \eqref{eq:Hurwitz} holds with 
some bilinear functions $z_\nu$ $(\nu=1,\dots,n)$ of the $x_i$ and the $y_j$ with coefficients in $\R$.
In order to exist such an equation, we need to have
\[
n\ge \max(r,s).
\]
It is a necessity condition, but not sufficient.
Finding an admissible triplet $(r,s,n)$ is called a Hurwitz problem,
and it is still an open problem
(cf.\ Rajwade~\cite{Rajwade},
see also Shapiro~\cite{Bk-Shapiro}).

\begin{Remark}
\label{rem:1}
Equation~\eqref{eq:Hurwitz} tells us that 
if a triplet $(r,s,n)$ is not admissible and if $r,s,n\ge 1$,
then there cannot exist a homogeneous cone satisfying
$(d_{32},d_{21},d_{31})=(r,s,n)$.
For example, it is known that if an irreducible homogeneous cone of rank $3$ satisfies
$d=d_{32}=d_{21}=d_{31}$, then we can only have $d=1,2,4,8$
(cf. Vinberg~\cite[\S8, Chapter III]{Vinberg}).
Moreover,
it is still an open problem whether
a triplet $(16,16,31)$ is admissible or not,
and hence 
so is whether there exists a homogeneous cone of rank $3$ with $(d_{32},d_{21},d_{31})=(16,16,31)$ or not.
\end{Remark}

\begin{Remark}
Vinberg~\cite[Definition 7, Chapter III]{Vinberg} reveals
relationship between
the structure of homogeneous cones
and the Hurwitz problem~\eqref{eq:Hurwitz} on quadratic forms
by stating that 
the off-diagonal spaces $\{\mathcal{V}_{kj}\}$ must satisfy
$\norm{X_{kj}X_{ji}}_{ki}=\norm{X_{kj}}_{kj}\norm{X_{ji}}_{ji}$
for any distinct integers $i,j,k\in\{1,\dots,r\}$.
There are some attempts to describe homogeneous cones of smaller dimensions
by specifying a product $X_{kj}X_{ji}$ explicitly
(see Kaneyuki--Tsuji~\cite{KT74} and Yamasaki--Nomura~\cite{YN} for example).
\end{Remark}

Conversely, if there exist bilinear functions $z_1,\dots,z_n$ of $x_i$, $y_j$
satisfying \eqref{eq:Hurwitz},
then we can construct two homogeneous cones of rank $3$, one of which satisfies
\[
(\dim \mathcal{V}_{32},\dim \mathcal{V}_{21}, \dim \mathcal{V}_{31})=(r,s,n).\]
These two homogeneous cones are mutually dual.
Assume that a triplet $(r,s,n)$ satisfies \eqref{eq:Hurwitz} with $r,s,n\ge 1$.
Let us put
$\bs{x}={}^t(x_1,\dots,x_r)\in\R^r$,
$\bs{y}={}^t(y_1,\dots,y_s)\in\R^s$ and 
$\bs{z}={}^t(z_1,\dots,z_n)\in\R^n$.
By assumption, 
there exist matrices 
$L(\bs{x})\in\mathrm{Mat}(n,s;\,\R)$ and  
$R(\bs{y})\in\mathrm{Mat}(n,r;\,\R)$ such that
$L(\bs{x})$ (resp.\ $R(\bs{y})$) is bilinear in the $x_i$ (resp.\ $y_i$) satisfying
\[
\bs{z}=L(\bs{x})\bs{y}=R(\bs{y})\bs{x}.
\]
For $m=r,s,n$,
we denote by $\innV{\cdot}{\cdot}$ the standard inner product of $\R^m$,
and by $\norm{\cdot}$ its induced norm.
Then,
by \eqref{eq:Hurwitz}, we have
\begin{equation}
\label{eq:LR}
{}^tL(\bs{x})L(\bs{x})=\norm{\bs{x}}^2I_s,\quad
{}^tR(\bs{y})R(\bs{y})=\norm{\bs{y}}^2I_r.
\end{equation}
For a symmetric matrix $X\in\mathrm{Sym}(N,\R)$,
we write $X\gg0$ if $X$ is positive definite.

\begin{Proposition}
\label{prop:cone-rsn}
Let $(r,s,n)$ be an admissible triplet with $r,s,n\ge 1$.
\begin{enumerate}[{\rm(1)}]
\item
Let $\Omega$ be an open convex cone defined by $\Omega=\set{X\in V}{X\gg0}$ where 
\begin{equation}
\label{eq:cone-rsn}
V:=\set{
X=\pmat{
	x_{11}I_n&R(\bs{y})&\bs{z}\\
	{}^{t\!}R(\bs{y})&x_{22}I_r&\bs{x}\\
	{}^{t\!}\bs{z}&{}^{t\!}\bs{x}&x_{33}
}%
}{
	\begin{array}{l}
	x_{11},x_{22},x_{33}\in\R\\
	\bs{x}\in\R^r,\ 
	\bs{y}\in\R^s,\ 
	\bs{z}\in\R^n
	\end{array}
}.
\end{equation}
Then, $\Omega$ is a homogeneous cone satisfying
$(\dim \mathcal{V}_{32},\dim \mathcal{V}_{21},\dim \mathcal{V}_{31})=(r,s,n)$.
A determinant of $X\in\Omega$ is given as
\begin{align*}
\det X
&=x_{11}^{n-r-1}
(x_{11}x_{22}-\norm{\bs{y}}^2)^{r-1}\\
&\quad\times
\left((x_{11}x_{22}-\norm{\bs{y}}^2)(x_{11}x_{33}-\norm{\bs{z}}^2)-\norm{x_{11}\bs{x}-{}^{t\!}R(\bs{y})\bs{z}}^2\right).
\end{align*}
\item
Let $\Omega'$ be an open convex cone defined by $\Omega':=\set{\Xi\in V'}{\Xi\gg0}$ where
\begin{equation}
\label{eq:dual-cone-rsn}
V':=
\set{\Xi=
	\pmat{\xi_{11}&{}^{t\!}\bs{\eta}&{}^{t\!}\bs{\zeta}\\
		\bs{\eta}&\xi_{22}I_s&{}^{t\!}L(\bs{\xi})\\
		\bs{\zeta}&L(\bs{\xi})&\xi_{33}I_n
}}{
	\begin{array}{l}
		\xi_{11},\xi_{22},\xi_{33}\in\R\\
		\bs{\xi}\in\R^r,\ \bs{\eta}\in\R^s,\ \bs{\zeta}\in\R^n
	\end{array}
}.
\end{equation}
Then, $\Omega'$ is a homogeneous cone
on which the group $H'$ defined below acts transitively by 
$\rho'(h')\Xi:=h'\Xi{}^{\,t\!}h'$ $(h'\in H',\ \Xi\in V')$ where
\[
H':=\set{
	h'=\pmat{h_{11}&{}^{t\!}\bs{c}&{}^{t\!}\bs{b}\\0&h_{22}I_s&{}^{t\!}L(\bs{a})\\0&0&h_{33}I_n}
}{
	\begin{array}{l}
		h_{11},h_{22},h_{33}\ne 0\\
		\bs{a}\in\R^r,\ \bs{b}\in\R^s,\ \bs{c}\in\R^n
	\end{array}
}.
\]
A determinant of $\Xi\in\Omega'$ is given as
\begin{align*}
	\det\Xi
	&=
	\xi_{33}^{n-s-1}(\xi_{22}\xi_{33}-\norm{\bs{\xi}}^2)^{s-1}\\
	&\quad\times
	\left((\xi_{11}\xi_{33}-\norm{\bs{\zeta}}^2)(\xi_{22}\xi_{33}-\norm{\bs{\xi}}^2)
	-\norm{\xi_{33}\bs{\eta}-{}^{t\!}L(\bs{\xi})\bs{\zeta}}^2\right).
\end{align*}
\item The cones $\Omega$ and $\Omega'$ are mutually dual through the following coupling
\[
\coupling{X}{\Xi}:=x_{11}\xi_{11}+x_{22}\xi_{22}+x_{33}\xi_{33}
+2\innV{\bs{x}}{\bs{\xi}}+2\innV{\bs{y}}{\bs{\eta}}+2\innV{\bs{z}}{\bs{\zeta}}.
\]
Namely, one has
\begin{align*}
\Omega'&=\set{\Xi\in V'}{\coupling{X}{\Xi}>0\textrm{ \rm for any } 
X \in\overline{\Omega }\setminus\{0\}},\\
\Omega &=\set{X  \in V }{\coupling{X}{\Xi}>0\textrm{ \rm for any }
\Xi\in\overline{\Omega'}\setminus\{0\}}.
\end{align*}
\end{enumerate}
\end{Proposition}
\begin{proof}
(1) The conditions $(V1)$--$(V3)$ are almost trivial except for the case $\mathcal{V}_{21}$ in $(V3)$,
and it follows from \eqref{eq:LR}.
%
%
For calculating $\det X$,
we recall a basic determinant formula
\begin{equation}\label{eq:det-formula}
	\det\pmat{A&B\\C&D}=(\det A)\det(D-CA^{-1}B)\quad(\det A\ne 0).
\end{equation}
Using this formula twice,
we have by \eqref{eq:LR}
\begin{align*}
\det X&=
(\det (x_{11}I_n))\det\pmat{(x_{22}-x_{11}^{-1}\norm{\bs{y}}^2)I_s
&
\bs{x}-x_{11}^{-1}{}^{\,t\!}R(\bs{y})\bs{z}
\\
{}^t(\bs{x}-x_{11}^{-1}{}^{\,t\!}R(\bs{y})\bs{z})
&
x_{33}-x_{11}^{-1}\,\norm{\bs{z}}^2
}\\
&=
x_{11}^n(x_{22}-x_{11}^{-1}\norm{\bs{y}}^2)^s
\left(x_{33}-x_{11}^{-1}\,\norm{\bs{z}}^2-\frac{\norm{\bs{x}-x_{11}^{-1}{}^{\,t\!}R(\bs{y})\bs{z}}^2}{x_{22}-x_{11}^{-1}\norm{\bs{y}}^2}\right),
\end{align*}
and hence the assertion (1) is confirmed.

\noindent(2)
It is enough to check the following conditions, which are the upper triangular version of the matrix realization by $(V1)$--$(V3)$:
A collection of vector spaces $\mathcal{V}^*_{jk}\subset\mathrm{Mat}(n_j,n_k;\,\R)$
$(1\le j<k\le 3)$
satisfies the following conditions:
\begin{enumerate}
\item[(V1$^*$)] $X_{ij}\in \mathcal{V}^*_{ij}$, $X_{jk}\in \mathcal{V}^*_{jk}$ $\Rightarrow$ $X_{ij}X_{jk}\in \mathcal{V}^*_{ik}$ $(1\le i<j<k\le 3)$,
\item[(V2$^*$)] $X_{ik}\in \mathcal{V}^*_{ik}$, $X_{jk}\in \mathcal{V}^*_{jk}$ $\Rightarrow$ 
$X_{ik} {}^{t\!}X_{jk}\in \mathcal{V}^*_{ij}$
$(1\le i<j<k\le 3)$,
\item[(V3$^*$)] $X_{jk}\in \mathcal{V}^*_{jk}$ $\Rightarrow$ $X_{jk}{}^{\,t\!}X_{jk}\in \R I_{n_j}$
$(1\le j<k\le 3)$.
\end{enumerate}
It is almost evident except for the case $\mathcal{V}^*_{23}$ in $(V3^*)$,
and it follows from \eqref{eq:LR}.
Since the determinant of $\Xi\in\Omega'$ can be calculated by using \eqref{eq:det-formula}
in a similar way to $\det X$ $(X\in\Omega)$,
we omit it.

\noindent(3)
Associated with $X\in\Omega$ and $\Xi\in\Omega'$,
we introduce the following notations.
\begin{align*}
\tilde{X}
&=
\pmat{\tilde x_{22}I_s&\tilde{\bs{x}}\\{}^t\tilde{\bs{x}}&\tilde{x}_{33}}
:=
\pmat{(x_{22}-x_{11}^{-1}\norm{\bs{y}}^2)I_s
	&
	\bs{x}-x_{11}^{-1}{}^{\,t\!}R(\bs{y})\bs{z}
	\\
	{}^t(\bs{x}-x_{11}^{-1}{}^{\,t\!}R(\bs{y})\bs{z})
	&
	x_{33}-x_{11}^{-1}\,\norm{\bs{z}}^2
},
\\
\dtilde{x}_{33}
&:=
\tilde{x}_{33}-\frac{\norm{\tilde{\bs{x}}}^2}{\tilde{x}_{22}}=\frac{\det X}{x_{11}^{n-r}(x_{11}x_{22}-\norm{\bs{y}}^2)^r},
\\
\tilde{\Xi}
&:=
\pmat{\xi_{22}I_s&{}^tL(\bs{\xi})\\ L(\bs{\xi})&\xi_{33}I_n},
\\
\tilde{\xi}_{22}
&:=
\xi_{22}-\frac{\norm{\bs{\xi}}^2}{\xi_{33}},
\\
\dtilde{\xi}_{11}
&:=
\xi_{11}-\innV{\pmat{\bs{\eta}\\\bs{\zeta}}}{\tilde{\Xi}^{-1}\pmat{\bs{\eta}\\ \bs{\zeta}}}
=
\frac{\det \Xi}{\xi_{33}^{n-s}(\xi_{22}\xi_{33}-\norm{\bs{\xi}}^2)^s}.
\end{align*}
For 
\[
h=\pmat{I_n&0&0\\{}^tR(\bs{y}')&I_r&0\\ {}^t\bs{z}'&0&1}\in H\quad\text{and}\quad
h'=\pmat{1&{}^t\bs{y}'&{}^t\bs{z}'\\0&I_s&0\\0&0&I_n}\in H',
\]
we have by a direct computation
\[
\coupling{\rho(h)X}{\Xi}=\coupling{X}{\rho'(h')\Xi}\quad(X\in\Omega,\ \Xi\in\Omega').
\]
According to this formula,
we decompose $X\in\Omega$ and $\Xi\in\Omega'$ respectively as
\begin{align*}
X&=\pmat{I_n&0\\A&I_{r+1}}
\pmat{x_{11}I_n&0\\0&\tilde{X}}
\pmat{I_n&{}^{t\!}A\\0&I_{r+1}},\\
\Xi&=\pmat{1&{}^t\bs{B}\\0&I_{s+n}}
\pmat{\dtilde{\xi}_{11}&0\\0&\tilde{\Xi}}
\pmat{1&0\\ \bs{B}&I_{s+n}},
\end{align*}
where we set
\[
A:=\frac{1}{x_{11}}\pmat{{}^{\,t}R(\bs{y})\\ {}^t\bs{z}}\in\mathrm{Mat}(r+1,n;\,\R),\quad
\bs{B}:=\tilde{\Xi}^{-1}\pmat{\bs{\eta}\\ \bs{\zeta}}\in\R^{s+n}.
\]
Similarly, we have
with respect to $\tilde{X}$ and $\tilde{\Xi}$
\begin{align*}
\tilde{X}&=
\pmat{I_r&0\\ \tilde x_{22}^{-1}{}^{\,t}\tilde{\bs{x}}&1}
\pmat{\tilde{x}_{22}I_r&0\\0&\dtilde{x}_{33}}
\pmat{I_r&\tilde{x}_{22}^{-1}\tilde{\bs{x}}\\ 0&1},
\\
\tilde{\Xi}&=
\pmat{I_s&\xi_{33}^{-1}{}^{\,t}L(\bs{\xi})\\0&I_n}
\pmat{\tilde{\xi}_{22}I_s&0\\0&\xi_{33}I_n}
\pmat{I_s&0\\\xi_{33}^{-1}L(\bs{\xi})&I_n}.
\end{align*}
Therefore,
if we set
\[
\bs{C}:=\bs{B}+\frac{1}{x_{11}}\pmat{\bs{y}\\ \bs{z}}\in\R^{s+n},
\quad
\bs{d}:=\frac{1}{\xi_{33}}\bs{\xi}+\frac{1}{\tilde{x}_{22}}\tilde{\bs{x}},
\]
%
then we have
\begin{align*}
\coupling{X}{\Xi}
&=
\coupling{\pmat{x_{11}&0\\0&\tilde X}}%
{\pmat{1&{}^t\bs{C}\\0&I_{s+n}}%
\pmat{\dtilde{\xi}_{11}&0\\0&\tilde \Xi}%
\pmat{1&0\\ \bs{C}&I_{s+n}}}\\
&=
x_{11}(\dtilde{\xi}_{11}+{}^{\,t\!}\bs{C}\tilde\Xi\bs{C})+\bigl(\,\tilde{X}\bigm|\tilde{\Xi}\,\bigr)\\
&=
x_{11}(\dtilde{\xi}_{11}+{}^{\,t\!}\bs{C}\tilde\Xi\bs{C})\\
&\qquad
+
\coupling{\pmat{\tilde x_{22}I_r&0\\0&\dtilde{x}_{33}}}{%
\pmat{I_s&{}^tL(\bs{d})\\0&I_n}%
\pmat{\tilde\xi_{22}I_s&0\\0&\xi_{33}I_n}%
\pmat{I_s&0\\L(\bs{d})&I_n}
}\\
&=
x_{11}(\dtilde{\xi}_{11}+{}^{\,t\!}\bs{C}\tilde\Xi\bs{C})
+
\tilde x_{22}\left(\tilde\xi_{22}+\xi_{33}\norm{\bs{d}}^2\right)
+
\dtilde{x}_{33}\xi_{33}\\
&=
x_{11}\dtilde{\xi}_{11}+\tilde x_{22}\tilde\xi_{22}+\dtilde{x}_{33}\xi_{33}
+
x_{11}{}^{\,t}\bs{C}\tilde\Xi\bs{C}+\tilde x_{22}\xi_{33}\norm{\bs{d}}^2.
\end{align*}
This equation shows that the cones $\Omega$ and $\Omega'$ are mutually dual with respect to the coupling $\coupling{\cdot}{\cdot}$.
\end{proof}

We note that,
since matrices $L(\bs{\xi}){}^{\,t}L(\bs{\xi})$ $(\bs{\xi}\in\R^r)$ and $R(\bs{y}){}^{\,t}R(\bs{y})$ $(\bs{y}\in\R^s)$ may be scalar multiples of the identity matrix $I_n$,
the following factors
\begin{align}
&(x_{11}x_{22}-\norm{\bs{y}}^2)(x_{11}x_{33}-\norm{\bs{z}}^2)-\norm{x_{11}\bs{x}-{}^{t\!}R(\bs{y})\bs{z}}^2,\label{factor:1}\\
&(\xi_{11}\xi_{33}-\norm{\bs{\zeta}}^2)(\xi_{22}\xi_{33}-\norm{\bs{\xi}}^2)
	-\norm{\xi_{33}\bs{\eta}-{}^{t\!}L(\bs{\xi})\bs{\zeta}}^2\label{factor:2}
\end{align}
of $\det X$ and $\det \Xi$ respectively are not necessarily irreducible.

\begin{Remark}
Let us give an explicit description of a homogeneous cone with respect to $(r,s,n)=(3,5,7)$.
In this case, we can take for example
\[
\begin{split}
&(x_1^2+x_2^2+x_3^2)(y_1^2+y_2^2+y_3^2+y_4^2+y_5^2)\\
&\quad=
(x_1y_1+x_2y_4-x_3y_3)^2+(x_1y_2-x_2y_3-x_3y_4)^2+(x_1y_3+x_2y_2+x_3y_1)^2\\
&\qquad+(x_1y_4-x_2y_1+x_3y_2)^2+(x_1y_5)^2+(x_2y_5)^2+(x_3y_5)^2
\end{split}
\]
so that 
\[
L(\bs{x})
=
\pmat{
x_1&0&-x_3&x_2&0\\
	0&x_1&-x_2&-x_3&0\\
	x_3&x_2&x_1&0&0\\
	-x_2&x_3&0&x_1&0\\
	0&0&0&0&x_1\\
	0&0&0&0&x_2\\
	0&0&0&0&x_3	
},\quad
R(\bs{y})
=
\pmat{
y_1&y_4&-y_3\\
y_2&-y_3&-y_4\\
y_3&y_2&y_1\\
y_4&-y_1&y_2\\
y_5&0&0\\
0&y_5&0\\
0&0&y_5
}.
\]
Thus, 
the cone $\Omega$ defined by 
\[
\Omega=\left\{\left(\begin{array}{*{7}{c}|*{3}{c}|c}
x_{11}&0&0&0&0&0&0		&y_1&y_4&-y_3	&z_1\\
0&x_{11}&0&0&0&0&0		&y_2&-y_3&-y_4	&z_2\\
0&0&x_{11}&0&0&0&0		&y_3&y_2&y_1		&z_3\\
0&0&0&x_{11}&0&0&0		&y_4&-y_1&y_2	&z_4\\
0&0&0&0&x_{11}&0&0		&y_5&0&0			&z_5\\
0&0&0&0&0&x_{11}&0		&0&y_5&0			&z_6\\
0&0&0&0&0&0&x_{11}		&0&0&y_5			&z_7\\ \hline
y_1&y_2&y_3&y_4&y_5&0&0		&x_{22}&0&0	&x_1\\
y_4&-y_3&y_2&-y_1&0&y_5&0	&0&x_{22}&0	&x_2\\
-y_3&-y_4&y_1&y_2&0&0&y_5	&0&0&x_{22}	&x_3\\ \hline
z_1&z_2&z_3&z_4&z_5&z_6&z_7	&x_1&x_2&x_3&x_{33}
\end{array}\right)
\gg0\right\}
\]
is a homogeneous cone satisfying 
$(d_{32},d_{21},d_{31})=(3,5,7)$,
and the cone $\Omega'$ defined by
{\small
\[
\Omega'=\left\{
\left(
	\begin{array}{c|*{5}{c}|*{7}{c}}
	\xi_{11}	&\eta_1&\eta_2&\eta_3&\eta_4&\eta_5		&\zeta_1&\zeta_2&\zeta_3&\zeta_4&\zeta_5&\zeta_6&\zeta_7\\ \hline
	\eta_1	&\xi_{22}&0&0&0&0	&\xi_1&0&\xi_3&-\xi_2&0&0&0\\
	\eta_2	&0&\xi_{22}&0&0&0	&0&\xi_1&\xi_2&\xi_3&0&0&0\\
	\eta_3	&0&0&\xi_{22}&0&0	&-\xi_3&-\xi_2&\xi_1&0&0&0&0\\
	\eta_4	&0&0&0&\xi_{22}&0	&\xi_2&-\xi_3&0&\xi_1&0&0&0\\
	\eta_5	&0&0&0&0&\xi_{22}	&0&0&0&0&\xi_1&\xi_2&\xi_3\\ \hline
	\zeta_1		&\xi_1&0&-\xi_3&\xi_2&0		&\xi_{33}&0&0&0&0&0&0\\
	\zeta_2		&0&\xi_1&-\xi_2&-\xi_3&0	&0&\xi_{33}&0&0&0&0&0\\
	\zeta_3		&\xi_3&\xi_2&\xi_1&0&0		&0&0&\xi_{33}&0&0&0&0\\
	\zeta_4		&-\xi_2&\xi_3&0&\xi_1&0		&0&0&0&\xi_{33}&0&0&0\\
	\zeta_5		&0&0&0&0&\xi_1				&0&0&0&0&\xi_{33}&0&0\\
	\zeta_6		&0&0&0&0&\xi_2				&0&0&0&0&0&\xi_{33}&0\\
	\zeta_7		&0&0&0&0&\xi_3				&0&0&0&0&0&0&\xi_{33}
\end{array}
\right)\gg0
\right\}.
\]
}
is linearly isomorphic to the dual cone of $\Omega$.
\end{Remark}

Let $(r,s,n)$ be an admissible triplet 
and $\Omega$ a homogeneous cone of dimensions 
$(d_{32},d_{21}, d_{31})=(r,s,n)$.
We consider degrees of basic relative invariants $\Delta_1,\Delta_2,\Delta_3$ of $\Omega$ and 
those $\Delta^*_1,\Delta^*_2,\Delta^*_3$ of the dual cone $\Omega^*$ of $\Omega$.
For simplicity,
we put $d_i=\deg \Delta_i$ and $d_i^*=\deg\Delta^*_i$ $(i=1,2,3)$.
Since switching $r$ and $s$ corresponds to taking a dual cone,
we can assume $r\le s$ without loss of generality.
For a positive integer $n$,
the Hurwitz-Radon number $\rho(n)$ of $n$ is defined by
\[
\rho(n)=8\alpha+2^\beta,
\]
where we write $n=2^{4\alpha+\beta}(2\ell+1)$ with $\alpha,\beta,\ell\in\mathbb{Z}_{\ge0}$ $(0\le \beta\le 3)$ uniquely.
This number $\rho(n)$ is the maximum $r$ such that a triplet $(r,n,n)$ is admissible
(cf.\ Rajwade~\cite[Theorem 10.1]{Rajwade}).

\begin{Proposition}
Let $\Omega$ be an irreducible homogeneous cone satisfying
$(d_{32},d_{21}, d_{31})=(r,s,n)$ with $r\le s$.
Then, there are four cases on the degrees $d_i$ and $d^*_i$ $(i=1,2,3)$.
\begin{enumerate}[\rm(1)]
\item \label{item1} 
The case $r=s=n$.
In this case, one has $n=1,2,4,8$ and $\Omega$ is a symmetric cone.
The degrees are given as
\[
(d_1,d_2,d_3)=(1,2,3),\quad
(d_1^*,d_2^*,d_3^*)=(3,2,1).
\]
\item \label{item2} 
The case $1\le r<s=n$.
In this case, one has $r\le\rho(n)$ and the degrees are given as
\[
(d_1,d_2,d_3)=(1,2,4),\quad
(d_1^*,d_2^*,d_3^*)=(3,2,1).
\]
\item \label{item3} 
The case $1\le r\le s<n$.
In this case, the degrees are given as
\[
(d_1,d_2,d_3)=(1,2,4),\quad
(d_1^*,d_2^*,d_3^*)=(4,2,1).
\]
\item \label{item4} 
The case $r=0$.
In this case, one has $s,n\ge 1$ and the degrees are given as
\[
(d_1,d_2,d_3)=(1,2,2),\quad
(d_1^*,d_2^*,d_3^*)=(3,1,1).
\]
\end{enumerate}
\end{Proposition}
\begin{proof}
We divide cases by
the number of zeros
and by the number of equal values among $r,s,n$.
We note that two or more of $r,s,n$ cannot be zero
by irreducibility.
We first assume $r=0$ so that $s,n\ge 1$.
In this case, 
by similar arguments of Proposition~\ref{prop:cone-rsn},
the cone $\Omega$ can be described as
\[
\Omega=\set{
X=\left(\begin{array}{cc|c|c}
x_{11}I_s&0&\bs{y}&0\\
0&x_{11}I_n&0&\bs{z}\\ \hline
{}^{t\!}\bs{y}&0&x_{22}&0\\ \hline
0&{}^t\bs{z}&0&x_{33}
\end{array}\right)
}{\begin{array}{l}
x_{11},x_{22},x_{33}\in\R\\
\bs{y}\in\R^s,\ \bs{z}\in\R^n\\
X\gg 0
\end{array}}
\]
and the dual cone of $\Omega$ is linearly isomorphic to
\[
\Omega'=\set{\Xi=\left(\begin{array}{c|c|c}
\xi_{11}&{}^t\bs{\eta}&{}^t\bs{\zeta}\\ \hline
\bs{\eta}&\xi_{22}I_s&0\\ \hline
\bs{\zeta}&0&\xi_{33}I_n
\end{array}\right)}{\begin{array}{l}
\xi_{11},\xi_{22},\xi_{33}\in\R\\
\bs{\eta}\in\R^s,\ \bs{\zeta}\in\R^n\\
\Xi\gg0
\end{array}}.
\]
By using these expression,
we see that 
\[
\Delta_1(X)=x_{11},\quad
\Delta_2(X)=x_{11}x_{22}-\norm{\bs{y}}^2,\quad
\Delta_3(X)=x_{11}x_{33}-\norm{\bs{z}}^2
\]
and
\[
\Delta'_1(\Xi)=\xi_{11}\xi_{22}\xi_{33}-\xi_{22}\norm{\bs{\zeta}}^2-\xi_{33}\norm{\bs{\eta}}^2,\quad
\Delta'_2(\Xi)=\xi_{22},\quad
\Delta'_3(\Xi)=\xi_{33}.
\]
These calculations show the assertion \eqref{item4}.

Next we consider the case $r\ne 0$.
In this case, we know that $(r,s,n)$ is an admissible triplet.
Let us divide cases by the number of equal values among $r,s,n$.
If all of the three coincide, 
then Vinberg~\cite[\S8, Chapter~III]{Vinberg} 
tells us that
$n=1,2,4,8$ and $\Omega$ is a symmetric cone.
In this case, 
the degrees of its basic relative invariants are well known (see \cite{FK94}, for example).
Hence, the assertion~\eqref{item1} is shown.
Next, suppose that exactly two of $r,s,n$ are the same.
There are two cases, that is, $r<s=n$ and $r=s<n$.
At first, we assume $r<s=n$ so that $r\le \rho(n)$.
We use Lemma~\ref{lemma:alg}.
In this case, we have $\bs{d}_1=\smat{0\\s\\n}=\smat{0\\n\\n}$ and $\bs{d}_2=\smat{0\\0\\r}$.
Then,
\[
\bs{l}^{(1)}_1=\bs{d}_1,\quad
\bs{l}^{(2)}_1=\bs{d}_1-\bs{d}_2=\pmat{0\\n\\n-r},\quad
\bs{l}^{(2)}_2=\bs{d}_2=\pmat{0\\0\\r},
\]
and hence we have $\varepsilon_{21}=1$, $\varepsilon_{31}=1$ and $\varepsilon_{32}=1$.
This implies that
\[
\sigma=\mathcal{E}_2\mathcal{E}_1=
\pmat{1&0&0\\0&1&0\\0&1&1}
\pmat{1&0&0\\1&1&0\\1&0&1}
=
\pmat{1&0&0\\1&1&0\\2&1&1}.
\]
For the dual cone, we use the same algorithm for 
$\bs{d}_1^*=\smat{n\\r\\0}$ and $\bs{d}_2^*=\smat{s\\0\\0}=\smat{n\\0\\0}$.
In this case, we have
\[
\bs{l}^{(1)*}_1=\bs{d}_1^*,\quad
\bs{l}^{(2)*}_1=\bs{d}^*_1-\bs{d}_2^*
=\pmat{0\\r\\0},\quad
\bs{l}^{(2)*}_2=\bs{d}_2^*=\pmat{s\\0\\0}
\]
and hence $\varepsilon^*_{13}=0$, $\varepsilon^*_{23}=1$ and $\varepsilon^*_{12}=1$.
Thus, $\sigma_*$ is calculated as
\[
\sigma_*=\mathcal{E}^*_2\mathcal{E}^*_1
=
\pmat{1&1&0\\0&1&0\\0&0&1}
\pmat{1&0&0\\0&1&1\\0&0&1}
=
\pmat{1&1&1\\0&1&1\\0&0&1}.
\]
This shows the assertion \eqref{item2}.

Next consider the case $1\le r=s<n$.
As a similar calculation as above,
we have $\bs{d}_1=\smat{0\\s\\n}=\smat{0\\r\\n}$, $\bs{d}_2=\smat{0\\0\\r}$ and
\[
\bs{l}^{(1)}_1=\bs{d}_1,\quad
\bs{l}^{(2)}_1=\bs{d}_1-\bs{d}_2=\pmat{0\\r\\n-r},\quad
\bs{l}^{(2)}_2=\bs{d}_2=\pmat{0\\0\\r},
\]
and hence $\varepsilon_{21}=1$, $\varepsilon_{31}=1$ and $\varepsilon_{32}=1$.
This implies that
\[
\sigma=\mathcal{E}_2\mathcal{E}_1=
\pmat{1&0&0\\0&1&0\\0&1&1}
\pmat{1&0&0\\1&1&0\\1&0&1}
=
\pmat{1&0&0\\1&1&0\\2&1&1}.
\]
The case $1\le r<s<n$ can be proceeded in the same lines for $\sigma$.
Moreover, we can calculate $\sigma_*$ similarly to obtain
$\sigma_*=\smat{1&1&2\\0&1&1\\0&0&1}$.
Therefore, we have shown the assertion~\eqref{item3}.
The proof is now completed.
\end{proof}


Let us explain what happens in the case (2), that is, $1\le r<s=n$.
Since $R(\bs{y})$ is an $n\times r$ matrix with $r<n$,
we see that
\[
R(\bs{y}){}^{\,t\!}R(\bs{y})\ne\norm{\bs{y}}^2I_n
\iff
\norm{{}^{t\!}R(\bs{y})\bs{z}}^2\ne\norm{\bs{y}}^2\norm{\bs{z}}^2.
\]
Therefore, the factor~\eqref{factor:1} of $\det X$ $(X\in\Omega)$ is irreducible and hence
the basic relative invariants $\Delta_i(X)$ $(i=1,2,3)$ of $\Omega$ are given as
\[
\begin{array}{c}
	\Delta_1(X)=x_{11},\quad
	\Delta_2(X)=x_{11}x_{22}-\norm{\bs{y}}^2,\\[0.5em]
	\Delta_3(X)=(x_{11}x_{22}-\norm{\bs{y}}^2)(x_{11}x_{33}-\norm{\bs{z}}^2)-\norm{x_{11}\bs{x}-{}^{t\!}R(\bs{y})\bs{z}}^2.
\end{array}
\]
On the other hand,
$L(\bs{\xi})$ is a square matrix of size $n$,
and the construction of $L(\bs{\xi})$ as in the book \cite[Theorem 10.1]{Rajwade}
implies that the matrix $\norm{\bs{\xi}}^{-1}L(\bs{\xi})$ is orthogonal and it satisfies
\[
\norm{L(\bs{\xi})\bs{\zeta}}^2=\norm{\bs{\xi}}^2\norm{\bs{\zeta}}^2.
\]
This means that
the factor~\eqref{factor:2} of $\det \Xi$ $(\Xi\in\Omega')$ is reducible and hence 
the basic relative invariants $\Delta'_i(\Xi)$ $(i=1,2,3)$ of $\Omega'$ are given as
\[
\begin{array}{c}
\Delta'_1(\Xi)=\xi_{11}\xi_{22}\xi_{33}+2\innV{\bs{\eta}}{L(\bs{\xi})\bs{\zeta}}-\xi_{11}\norm{\bs{\xi}}^2-\xi_{22}\norm{\bs{\zeta}}^2-\xi_{33}\norm{\bs{\eta}}^2,\\[0.5em]
\Delta'_2(\Xi)=\xi_{22}\xi_{33}-\norm{\bs{\xi}}^2,\quad
\Delta'_3(\Xi)=\xi_{33}.
\end{array}
\]


\bibliographystyle{abbrv}
\bibliography{bibliography}


\end{document}